\documentclass[12pt]{article}

\usepackage{amsmath,amsthm,amsfonts}
\usepackage{fullpage}

\input{xy}
\xyoption{all}
\xyoption{matrix}
\usepackage{arydshln}

\def\GL{\mathrm{GL}}\def\SL{\mathrm{SL}}
\def\End{\mathrm{End}}\def\Hom{\mathrm{Hom}}\def\Ker{\mathrm{Ker}}

\def\Der{\mathrm{Der}}\def\bider{\mathrm{BiDer}}
\def\Bider{\mathrm{BiDer}}\def\BiDer{\mathrm{BiDer}}
\def\Inder{\mathrm{InnDer}}\def\InDer{\mathrm{InnDer}}\def\CoDer{\mathrm{CoDer}}
\def\Alt{\mathrm{Alt}}

\def\D{\mathcal{D}}
\def\hh{\mathrm{height}}\def\ra{\overline}\def\sub{\underline}
\def\w{\wedge}\def\el{\ell}
\def\wt{\widetilde}\def\ot{\otimes}
\def\ad{\mathrm{ad}}\def\id{\mathrm{Id}}\def\Id{\mathrm{Id}}
\def\ie{{\em{i.e. }}}\def\Om{\Omega}

\def\b{\mathfrak{b}}\def\l{\mathfrak{L}}\def\r{\mathfrak{r}}
\def\u{\mathfrak{u}}\def\U{\mathcal{U}}\def\g{\mathfrak{g}}
\def\s{\mathfrak{s}}\def\n{\mathfrak{n}}
\def\h{\mathfrak{h}}\def\su{\mathfrak{su}}
\def\sl{\mathfrak{sl}}\def\gl{\mathfrak{gl}}\def\aff{\mathfrak{aff}}
\def\Z{\mathcal{Z}}\def\ZZ{\mathbb{Z}}
\def\k{\mathbb{K}}\def\K{\mathbb{K}}\def\E{\mathbb{E}}
\def\R{\mathbb{R}}\def\C{\mathbb{C}}\def\DD{\mathbb{D}}

\theoremstyle{plain} 
\newtheorem{theorem}{Theorem}[section] 
 
\newtheorem{prop}[theorem]{Proposition} 
 
\newtheorem{teo}[theorem]{Theorem}
\newtheorem{coro}[theorem]{Corollary}
\newtheorem{lema}[theorem]{Lemma}

\theoremstyle{definition} 
\newtheorem{defi}[theorem]{Definition}
\newtheorem{ex}[theorem]{Example}
\newtheorem{obs}[theorem]{Remark}
\newtheorem{rem}[theorem]{Remark}

\title{Trivial Central Extensions of Lie Bialgebras}
 
\author{Marco A. Farinati\thanks{Member of CONICET. Partially supported by UBACyT X051, PICT
2006-00836 and MathAmSud 10-math-01 OPECSHA, mfarinat@dm.uba.ar.}
and A. Patricia Jancsa
\thanks{Partially supported by UBACyT X051, PICT 2006-00836 and MathAmSud 10-math-01 OPECSHA,
pjancsa@dm.uba.ar.}
}

\begin{document}
\maketitle

\noindent {\small Dto. de Matem\'atica, FCEyN, Universidad de Buenos Aires\\
Pabell\'on I, Cdad. Universitaria, 1428 
Buenos Aires, Argentina.}

\begin{abstract} From a Lie algebra $\g$ satisfying $\Z(\g)=0$ and
$\Lambda^2(\g)^\g=0$ (in particular, for $\g$ semisimple) we describe explicitly all 
 Lie bialgebra structures on extensions of the form $\l =\g\times \K$ in terms of 
 Lie bialgebra structures on $\g$ (not necessarily factorizable nor quasi-triangular) 
and its biderivations, for any field $\K$ with char $\K=0$. If moreover, $[\g,\g]=\g$,
 then we describe also all  Lie bialgebra structures on extensions $\l =\g\times \K^n$.
In interesting cases we characterize the Lie algebra of biderivations.
\end{abstract}

\section{Introduction and preliminaries}

Recall \cite{D1,D2} that a Lie bialgebra over a field $\K$ is a triple $(\g,[-,-],\delta)$
where $(\g,[-,-])$ is a Lie algebra over $\K$
 and $\delta:\g\to\Lambda^2\g$
is such that 
\begin{itemize}
 \item  
$\delta\colon \g\to \Lambda^2\g$
satisfies co-Jacobi identity, namely $\Alt((\delta\ot\id)\circ\delta)=0$,
 \item $\delta\colon \g\to\Lambda^2\g$ is a 1-cocycle in the Chevalley-Eilenberg
complex of the Lie algebra $(\g,[-,-])$ with coefficients in $\Lambda^2\g$.
\end{itemize}
In the finite dimensional case, 
$\delta\colon \g\to \Lambda^2\g$
satisfies co-Jacobi identity if and only if
the bracket defined by
$\delta^*\colon  \Lambda^2\g^*\to \g^*$
satisfies Jacobi identity.
In general,  co-Jacobi identity for $\delta$
is equivalent to the fact that the unique derivation of degree one
$\partial_\delta\colon \Lambda^\bullet \g\to \Lambda^\bullet \g$, whose restriction to $\g$ agrees with $\delta$, satisfies $\partial_\delta^2=0$.
We will usually denote a Lie bialgebra,  with underlying Lie algebra  $\g=(\g,[-,-])$, by $(\g,\delta)$.
A Lie bialgebra $(\g,\delta)$ is said a coboundary Lie bialgebra
if there exists $r\in\Lambda ^2\g$ such that $\delta (x)=\ad _x(r)$ $\forall x\in\g$; \ie
 $\delta =\partial r$ is a 1-coboundary in the Chevalley-Eilenberg complex with coefficients in $\Lambda ^2\g$.  Coboundary Lie bialgebras are denoted by $(\g, r)$, although $r$ is in general not unique.
We have that $r$ and $r'$ give rise to the same cobracket if and only if
$r-r'\in(\Lambda^2\g)^\g$, so $r$ is uniquely determined
by $\delta$ in the semisimple
case, since $(\Lambda^2\g)^\g=0$ for $\g$ semisimple.

Recall that $r\in\g\ot\g$ satisfies the {\em classical Yang-Baxter equation}, CYBE for short, if
\[
[r^{12}, r^{13}] + [r^{12}, r^{23}] + [r^{13}, 
r^{23}]=0,
\]
where the Lie bracket is taken in the repeated index; for example, if $r=\sum _{i}r_i\ot r^i$ then
$r^{12}:=r\ot 1$, $ r^{13}:=\sum _{i}r_i\ot 1\ot r^i$
and $r^{23}:=1\ot r\in \U(\g)^{\ot 3}$, so
$[r^{12}, r^{13}]=\sum _{i,j}[r_i,r_j]\ot r^i\ot r^j\in\g\ot\g\ot\g\hookrightarrow\U(\g)^{\ot 3}$, and so on for the other terms of CYBE.
We denote the left hand-side of CYBE by CYB($r$). 

If $r\in\Lambda ^2\g$, then $\delta =\partial r$ satisfies  co-Jacobi if and only if CYB$(r)\in\Lambda ^3\g$ is $\g$-invariant.
 If $(\g,r)$ is a coboundary Lie bialgebra and $r$ satisfies CYBE,
$(\g,r)$ is said {\em triangular}.
A Lie bialgebra is {\em quasi-triangular} if there exist $r\in\g\otimes \g$, not necessarily skewsymmetric, such that $\delta (x)=\ad _x(r)$ $\forall x\in\g$ and $r$ satisfies CYBE;
if, moreover,  the symmetric component of $r$ induces a nondegenerate inner
 product on $\g ^*$, then $(\g,\delta)$ is said {\em factorizable}
\cite{RS}.
Quasi-triangular Lie bialgebras are also denoted by $(\g, r)$, although $r$ is in general not unique. Nevertheless, in the semisimple case the skew symmetric component
$r_\Lambda$ of $r$ is uniquely determined by $\delta$.
A quasi-triangular Lie bialgebra $(\g,r)$ is, in particular, a coboundary Lie bialgebra, with
the coboundary chosen as the skewsymmetric component of $r$.

If $(\g,\delta)$ is a real Lie bialgebra, then $\g\ot_\R\C$ is a complex
Lie bialgebra with cobracket
$\delta\ot_\R\id_\C:\g\ot_\R\C\to (\Lambda^2_\R \g)\ot_\R\C\cong
\Lambda^2_\C (\g\ot_\R\C)$. A real Lie bialgebra is coboundary if and only if
its complexification is coboundary. On the other hand, it may happen
that $(\g\ot_\R\C,\delta\ot_\R\Id_\C)$ 
 is  factorizable but $(\g,\delta)$
is not; in this case we call it {\em almost factorizable}. 

\subsection*{The theorem of Belavin and Drinfeld}
Let $\g$ be a complex simple Lie algebra, $\Omega\in (S^2\g)^\g$ the {\em Casimir} element corresponding to a 
fixed non-degenerate, symmetric, invariant, bilinear form $(-,-)$ on $\g$, and let $\h \subset \g$ be a Cartan subalgebra. Let $\Delta$ be a
choice of a set of simple roots.
A Belavin-Drinfeld triple (BD-triple for short) is a triple $(\Gamma_1, \Gamma_2,\tau)$, where
$\Gamma_1, \Gamma_2$ are subsets of $\Delta$, and $\tau:\Gamma_1 \to\Gamma_2$ is a bijection that preserves the inner product
and satisfies the nilpotency condition: for any $\alpha\in\Gamma_1$, there exists a positive integer $n$ for
which $\tau ^n(\alpha)$ belongs to $\Gamma_2$ but not to $\Gamma_1$.
Let $(\Gamma_1, \Gamma_2,\tau)$ be a BD triple. Let 
$\widetilde\Gamma_i$ be the set of positive roots lying
in the subgroup generated by $\Gamma_i$, for $i = 1, 2$. There is an associated partial order on
$\Phi^+$ given by $\alpha\prec\beta$ if $\alpha\in\widetilde\Gamma_1,\ \beta\in\widetilde\Gamma_2$ and $\beta = \tau ^n(\alpha)$ for a positive integer $n$.
A continuous parameter for the BD triple $(\Gamma_1, \Gamma_2,\tau)$ is an element
$r_0\in\h\ot\h$ such that
$(\tau (\alpha)\ot\id +\id\ot\alpha)r_0=0\
\forall\alpha\in\Gamma_1$, and $r_0+r_0^{21}=\Omega_0$, the $\h\ot\h$-component of $\Omega$.

\begin{teo}\label{BD-teo}
(Belavin-Drinfeld, see \cite{BD}). Let $(\g,\delta)$ be a factorizable complex simple Lie
bialgebra. Then there exists a non-degenerate, symmetric, invariant, bilinear form on $\g$
with corresponding Casimir element $\Om$,
 a Cartan subalgebra $\h$, a system of simple roots 
$\Delta$, a BD-triple  $(\Gamma_1, \Gamma_2,\tau)$
and 
continuous parameter $r_0\in\h\ot\h$ 
such that
$\delta(x)=\ad_x(r)$ for all $x\in\g$, with $r$ given by
\begin{eqnarray}
r =
r_0
+\sum_{\alpha\in\Phi^+}x_{-\alpha}\ot x_{\alpha} 
+\sum_{\alpha\in\Phi^+:\alpha\prec\beta}x_{-\alpha}\wedge x_{\beta} 
\ \label{BD-r-matrix}
\end{eqnarray}
where $x_{\pm\alpha}\in\g_{\pm\alpha}$, $\pm\alpha\in\pm\Phi^+$ are root vectors normalized by
$(x_{\alpha},x_{-\alpha})=1$,
$\forall \pm\alpha\in\pm\Phi^+$,
Clearly, $r + r^{21} = 
\Omega$.

Reciprocally, any $r$  of the form given above satifies CYBE and endows the Lie algebra $\g$ of a factorizable Lie bialgebra structure. 
\end{teo}

The component 
$\sum_{\alpha\in\Phi^+}x_{-\alpha}\w x_\alpha +\Omega$
is called the {\em standard part} and it is denoted by $r_{st}$, so 
$r=
r_{st}+\sum_{\alpha\prec\beta}x_{-\alpha}\w x_\beta
+\lambda
$, if we decompose $r_0=\lambda +\Omega _0$, $\lambda\in\Lambda^2\h$.

\begin{obs}
Some authors have considered more general versions of the previous theorem
 (see \cite{H} and \cite{Del}  for the semisimple and reductive versions).
 In this work, 
 we give a new description for the reductive Lie bialgebras 
without using the previous works but starting from a given Lie bialgebra structure on the semisimple factor $\g$.
\end{obs}

Our point of view is the following: From a Lie algebra $\g$ over  a field $\K$ with char $\K=0$ satisfying $\Z(\g)=0$ and
$\Lambda^2(\g)^\g=0$ we describe explicitly all the Lie bialgebra structures on extensions of the form $\l =\g\times \K$ 
in terms of Lie bialgebra structures on $\g$ and its biderivations. 
If moreover, $[\g,\g]=\g$, then we describe also all the Lie bialgebra structures on extensions $\l =\g\times \K^d$ for any $d$.
In the semisimple factorizable case, the Lie bialgebra structures on $\g$ are known \cite{BD,AJ,Del}; we make a detailed analysis of the 
biderivations in this case and give an alternative description of the extensions to reductive Lie bialgebras. This characterization includes the reductive factorizable case, but actually we obtain all Lie bialgebra structures on $\l =\g\times \K^d$ that restrict to a given 
Lie bialgebra structure on $\g$, which include non-factorizable and even non-coboundary ones. The latter were not considered in  previous works.


\subsection*{The center and the derived ideal $ [\g, \g ]$}
The next statement is straight but useful.

\begin{prop} \label{central} 
Let $\l$ be a Lie algebra and $\delta:\l\to\Lambda ^2\l$
a 1-cocycle, then
\begin{enumerate}
\item $[\l,\l]$ is a coideal \ie $\delta [\l,\l]\subseteq [\l,\l]\w\l$. As a consequence, if $(\l,\delta)$ is a Lie bialgebra then
the quotient $\l / [\l, \l ]$ admits a unique Lie bialgebra structure such that the canonical
 projection is a Lie bialgebra map.
Moreover, if  $(\l,\delta _1)\cong (\l,\delta _2)$ as Lie bialgebras, then 
$(\l /[\l,\l],\ra{\delta }_1)\cong (\l /[\l,\l],\ra{\delta} _2)$.

\item If $\Z(\l)$ denotes the center of the Lie algebra $\l$, then
 $\delta (\Z(\l))\subseteq\Lambda ^2(\l)^{\l}$.
\end{enumerate}\end{prop}

\begin{proof} 1) It is enough to notice that for any $x,y\in\l$, $\delta [x,y]=\ad _x\delta y-\ad _y\delta x
\in [\l,\l]\w\l$.

2) 
If $z$ is central, then $[z,x]=0$  for all $x\in\l$, so for
 a  1-cocycle $\delta$ we get
\[0=\delta ([z,x])=[z,\delta (x)]+[\delta (z),x]=[\delta (z),x]
\]
Hence, $\ad_x\delta (z)=0$ for all $x\in\l$.
\end{proof}

\begin{coro}
If $\l$ is a Lie bialgebra such that $(\Lambda^2\l)^\l=0$ then $\Z(\l)$ is a coideal.
\end{coro}

\subsection*{1-cocycles in product algebras}

Let $\l=\g\times V$, where $\g$ is a Lie algebra over a field $\k$ and
$V$ is a $\k$-vector space, considered as abelian Lie algebra.
The second exterior power of $\l$ can be computed as
\[
\begin{array}{ccccl}
\Lambda^2\l&=&
\Lambda^2(\g\times V)&=&\left(
\Lambda^2\g\ot \Lambda^0 V\right)
\oplus 
\left(\Lambda^1\g \ot \Lambda^1 V \right)
\oplus 
\left(\Lambda^0\g \ot \Lambda^2 V \right)
\\
&&&\cong&
\Lambda^2\g\oplus 
\g \ot V
\oplus 
\Lambda^2 V 
\end{array}\]
Notice 
that this is an $\l$-module decomposition,
so
\[
H^1(\l,\Lambda^2\l)\cong
H^1(\l,\Lambda^2\g)\oplus 
H^1(\l,\g\ot V)
\oplus 
H^1(\l,\Lambda^2 V) 
\]
 Now we recall the 
K\"unneth
formula 
\[
H^1(\g\times V,M_1\ot M_2)\cong
H^1(\g,M_1)\ot  H^0(V, M_2)\oplus H^0(\g,M_1)\ot H^1(V,M_2)
\]
\[\cong
H^1(\g,M_1)\ot  M_2\oplus M_1^\g\ot \Hom(V,M_2)
\]
where we use the equality $H^0(\g,M)=M^\g$ for any $\g$-module $M$. We assume that $M_2$ is a
 trivial representation of $V$ (e.g. $M_2=\k$, $V^\ad$,
or  $\Lambda^2V$),  so $M_2^V=M_2$ and
 $H^1(V,M_2)=\Hom(V,M_2)$. If we applied the K\"unneth formula
in our case, we get
\[
\begin{array}{rcl}
H^1(\l,\Lambda^2\l)&=&
H^1(\g, \Lambda^2\g) \oplus (\Lambda^2\g)^\g\ot V^*
\\
&\oplus&
H^1(\g,\g)\ot V \oplus (\g)^\g\ot \Hom(V, V)
\\
&\oplus&
H^1(\g, \k)\ot \Lambda^2 V \oplus \Hom(V, \Lambda^2 V)
\end{array}
\]
Recalling that $H^1(\g,M)=\Der(\g,M)/\Inder(\g,M)$ and, in
particular,
\[H^1(\g,\k)=\Der(\g,\k)\cong (\g/[\g,\g])^*,\]
 we get the final formula:
\[
\begin{array}{rcl}
H^1(\l,\Lambda^2\l)&=&
H^1(\g, \Lambda^2\g) \oplus (\Lambda^2\g)^\g\ot V^*
\\
&\oplus&
\Der(\g,\g)/\Inder(\g,\g)\ot V \oplus \Z(\g)\ot \End(V)
\\
&\oplus&
(\g/[\g,\g])^*\ot \Lambda^2 V \oplus \Hom(V,\Lambda^2 V).
\end{array}
\]
We have the following special, favorable cases.
\begin{lema}
\label{lema:favorable}
Let $\l=\g\times V$ as before. 
\begin{enumerate}
\item If $\dim V=1$ then
\[
H^1(\l,\Lambda^2\l)\cong
H^1(\g, \Lambda^2\g) \oplus (\Lambda^2\g)^\g
\oplus \Der(\g,\g)/\Inder(\g,\g) \oplus \Z(\g).
\]
\item If $\g$ is semisimple, then $
H^1(\l,\Lambda^2\l)\cong \Hom(V,\Lambda^2 V)
$.
\item If $\dim V=1$ and $\g$ is semisimple, then $H^1(\l,\Lambda^\l)=0$, in particular,
every Lie bialgebra structure on $\l$ is coboundary.
\end{enumerate}
\end{lema}

\begin{ex}
If $\g=\su(2)$ or $\g=\sl(2,\R)$, then every 1-cocycle in $\g\times\R$ is  coboundary.
But this property does not hold for instance in $\sl(2,\R)\times\R^2$, or
$\gl(2,\R)\times \gl(2,\R)$.
\end{ex}

\subsection*{Extensions of scalars}

Let $\K\subset \E$ be a finite field extension, if $\g$ is a Lie (bi)algebra over $\K$, then
$\g\ot_\K \E$ is naturally a Lie (bi)algebra over $\E$ and $\Lambda^2_\E(\g \ot_\K \E)\cong
(\Lambda^2_\K\g )\ot_\K \E$. Let us denote by $H_\K^\bullet(\g,-)$ and 
$H_\E^\bullet(\g\ot_K\E,-)$ the Lie algebra cohomology of $\g$ as $\K$-Lie algebra and
of $\g\ot_\K\E$ as $\E$-Lie algebra, respectively.
 Since Lie cohomology extends scalars, i.e.  if $M$ is a $\g$-module
and we consider $M\ot_\K \E$ as $(\g\ot_K \E)$-module then
$H_\E^\bullet(\g\ot_\K\E,M\ot_\K \E)=
H_\K^\bullet(\g,M)\ot_\K\E$,  we have
$H_\E^\bullet(\g\ot_\K \E,M\ot_\K \E)=0\iff H_\K^\bullet(\g,M)=0$ and
$H_\K^\bullet(\g,M)$ identifies with a $\K$-vector subspace of
$H_\E^\bullet(\g\ot_\K \E,M\ot_\K \E)$. 
In particular,
if $(\g,\delta)$ is a $\R$-Lie bialgebra, then it is coboundary if
and only if its complexification is coboundary.

\section{Biderivations}

For a Lie bialgebra $(\g,\delta)$, a map $D\colon \g\to\g$ which is at the same time a derivation and a coderivation is said a {\em biderivation}.
The set of all biderivations of $(\g,\delta)$ is denoted by $\BiDer(\g)$.
For an  inner biderivation we understand a biderivation which is inner as a derivation.

\begin{defi}Let $(\g,\delta)$ be a Lie bialgebra; we consider the {\em characteristic map} $\D_\g\colon \g\to\g$ defined by
$\D_\g (x):=[\ ,\ ](\delta x)=[x_1,x_2]$ for any $x\in\g$, where we denote 
$\delta x=x_1\w x_2$ in Sweedler type notation. 
\end{defi}
This map contains much information of the Lie bialgebra and it will be useful along this work. 
When it is clear from the context, $\D_\g$ will be denoted by $\D$. Due to the next proposition, we will call $\D_\g$ the {\em characteristic biderivation} of $\g$.

\begin{prop}
\label{propD}
If $(\g,\delta)$ is a Lie bialgebra then its characteristic map $\D$ is both a derivation and a coderivation.
\end{prop}
\begin{proof}Let us see that $\D$ is a derivation. If $x,y\in\g$, then
\[
\begin{array}{rcl}
\D ([x,y])&=&[\ ,\ ](\delta [x,y])=[\ ,\ ](\ad _x\delta y-\ad _y\delta x)
\\
&=&[\ , \ ]\left([x,y_1]\w y_2+y_1\w [x,y_2]+[x_1,y]\w x_2+x_1\w [x_2,y]\right)
\\
&=&[[x,y_1], y_2]+[y_1, [x,y_2]]+[[x_1,y], x_2]+[x_1, [x_2,y]]
=[x,[y_1,y_2]]+[[x_1,x_2],y]
\\
&=&[x,\D y]+[\D x,y].
\end{array}
\]
Notice that for a finite dimensional Lie bialgebra $(\g,[\ ,\ ],\delta)$,
 once we know that $\D_\g$ is a derivation in $(\g,[\ ,\ ])$,
 $\D_{\g^*}$ is a derivation in $(\g^*,\delta ^*)$, thus $\D_{\g}$ 
a coderivation in $(\g,\delta )$, since $\D_{\g^*}=(\D_{\g})^*$. 
Alternatively, one may prove it directly:
\begin{eqnarray}
\delta(\D(x))&=&\delta([x_1,x_2])
=[\delta x_1,x_2]+[x_1,\delta x_2]
=[ x_{11}\w x_{12},x_2]+[x_1,x_{21}\w x_{22}]\nonumber
\\
&=&[ x_{11},x_2]\w x_{12}+ x_{11}\w [x_{12},x_2]
+[x_1,x_{21}]\w x_{22}
+x_{21}\w [x_1,x_{22}]\label{d-delta}
\end{eqnarray}
On the other hand, co-Jacobi identity for $\delta$ implies
\[
0=(\delta\ot 1-1\ot \delta)\delta(x)=
x_{11}\w x_{12}\w x_2-
x_1\w x_{21}\w x_{22}
\]
then
$x_{11}\w x_{12}\w x_2=x_1\w x_{21}\w x_{22}
$ and $x_{11}\w x_{2}\w x_{12}=x_1\w x_{22}\w x_{21}$;
hence
\[
[x_{11}, x_{2}]\w x_{12}=[x_1, x_{22}]\w x_{21}
\]
So, the first and the last terms of the four terms 
in  formula (\ref {d-delta}) cancel and we get
\[
\delta(\D(x))=\delta([x_1,x_2])
= x_{11}\w [x_{12},x_2]
+[x_1,x_{21}]\w x_{22};
\]
using co-Jacobi identity again, the last formula equals
\[
= x_{1}\w [x_{21},x_{22}]
+[x_{11},x_{12}]\w x_{2}=x_1\wedge \D(x_2)+\D(x_1)\wedge x_2
=(1\ot \D+\D\ot 1)\delta(x).
\]
\end{proof}
\begin{prop}
\label{hache-erre}
Let $\g$ be a coboundary Lie bialgebra and $r\in\Lambda ^2\g$
such that $\delta(x)=\ad_x(r)$; consider $H_r:=[-,-](r)\in\g$ and $\D_\g$ the characteristic
 biderivation of $\g$, then $\D_\g=-\ad_{H_r}$.
\end{prop}
\begin{proof}
Write $r=r_1\ot r_2$ in Sweedler-type notation, so for any $x\in\g$
\[
\begin{array}{rclcl}
\D _\g(x)=[- ,- ]\circ\delta(x)
&=&[- ,- ](\ad_x(r_1\ot r_2))
&=&[[x,r_1],r_2]+[r_1,[x,r_2]]\\
&=&[x,[r_1,r_2]]
&=&[x,H_r]\\
&=&-\ad_{H_r}(x)
\end{array}
\]
\end{proof}

\begin{prop}
Let $\g$ be a Lie bialgebra and $\D_\g$ its characteristic biderivation. If
$E\in\bider(\g)$ then $[\D,E]=0$.
\end{prop}
\begin{proof}
The definition of coderivation says that $E$ satisfies
$
(E\ot \id+\id\ot E)\delta =\delta E$;
on the other hand, since $E$ is a derivation, $E[x,y]=[Ex,y]+[x,Ey]$, in other words,
\[
[-,-](E\ot \id+\id\ot E)= E[-,-]\]
Both properties together imply
\[
\D_\g E=
[-,-]\delta E=
[-,-](E\ot \id+\id\ot E)\delta=
E[-,-]\delta=
E\D_\g\]
\end{proof}
\begin{coro}
\label{corobider}
Let $\g$ be a  Lie bialgebra such that $\D_\g$  is an inner biderivation; write $\D_\g=\ad_{H_0}$ for some $H_0\in\g$.
\begin{enumerate}
\item If $E$ is a biderivation, then $E(H_0)\in\Z\g$.
\item
If $E=\ad_{x}\in\bider$, then 
$[x,H_0]\in\Z\g$; in particular, if $\Z\g=0$ then $x$
commutes with $H_0$.
\end{enumerate}
\end{coro}
\begin{proof}
1. We know $[E,\D_\g]=0$, then for any $x\in\g$, 
\[
\begin{array}{rcl}
0&=&[E,\D_\g](x)
=[E,\ad_{H_0}](x)=E(\ad_{H_0}(x))-\ad_{H_0}(E(x))\\
&=&E([H_0,x])-[H_0,E(x)]\\
&=&
[E(H_0),x]+[H_0,E(x)]-[H_0,E(x)]
\\
&=&[E(H_0),x]
=\ad_{E(H_0)}(x)
\end{array}
\]
hence $E(H_0)\in\Ker(\ad)=\Z\g$.
The second statement is a direct consequence of the first.
\end{proof}

We quote a result from \cite{AJ}, which together with the previous corollary
implies a very interesting fact.

\begin{prop}\label{centr=cartan}
If $\g$ is semisimple and $(\g, r)$ is an (almost) factorizable Lie bialgebra, 
then $H_r :=[-,-](r)$ is a regular element and so 
$\h :=\Z_\g(H_r)$, the centralizer of $H_r$, is a Cartan subalgebra of $\g$.
\end{prop}
\begin{proof}
 This statement is proved for  both, real and complex,  simple cases  in \cite{AJ}, but
 the proof remains valid {\em mutatis mutandis}
for the semisimple case.
\end{proof}

\begin{coro}\label{bider-semi}
Any biderivation of a  factorizable semisimple Lie bialgebra $(\g, r)$ is of the form  $\ad_H$ with $H\in \h=\Z_\g(H_r)$.
In particular,  $\bider(\g,\delta)$ is an abelian Lie algebra.
\end{coro}
\begin{proof}
If $\g$ is semisimple then every derivation is inner and
$\Z\g=0$, so $E=\ad_{x_0}$ and $x_0$ commutes with $H_r$. In particular, $x_0$ belongs
to the centralizer of $H_r$.
\end{proof}

Another characterization of inner biderivations is the following.

\begin{prop} 
\label{propbider}
Let $(\g,\delta)$ be a Lie bialgebra and $D=\ad _{x_0}$ an inner derivation, then $D$ is a coderivation if and only if $\delta x_0\in(\Lambda ^2\g)^{\g}$. 
In particular, if $(\Lambda ^2\g)^{\g}=0$,  then 
the map $x_0\mapsto \ad_{x_0}$ induces an isomorphism of Lie algebras
$\Ker \delta /\left(\Z(\g)\cap \Ker \delta\right)\cong\InDer(\g)\cap\CoDer(\g)$.
\end{prop}

\begin{proof}
 By definition, $D$ is a coderivation if and only if
$(D\ot\id +\id\ot D)\circ\delta =\delta \circ D$. 
Since $D=\ad _{x_0}$, we have $(D\ot\id +\id\ot D)(x\ot y)=\ad_{x_0}(x\ot y)$.
So, the coderivation condition reads
\[
\delta[x_0,z]=\ad_{x_0}\delta (z)
\]
for all $z\in\g$. On the other hand, $\delta$ is a 1-cocycle, namely
\[
\delta[x_0,z]=\ad_{x_0}\delta (z)-\ad_z\delta ({x_0})
\]
Hence, $D$ is a coderivation if and only if $\ad_z\delta ({x_0})=0$ 
for all $z\in\g$.
\end{proof}

\begin{coro}
\label{coro:bider}
Let $(\g,\delta)$ be a Lie bialgebra such that every derivation is inner, $\Z\g=0$ and $(\Lambda ^2\g)^{\g}=0$,
then $\BiDer(\g,\delta)\cong\Ker \delta$. 
In particular, if $\g$ is semisimple then the result holds.
\end{coro}

\begin{ex}\label{ex:aff2}
The  non-commutative two dimensional Lie algebra 
$\g=\aff_2(\K)$, verifies $\Der(\g)=\Inder(\g)$, $\Z(\g)=0$ and $(\Lambda ^2\g)^{\g}=0$ but it is not semisimple.
In fact, this is the
``$\sl_2$-case'' of the general and classical result (see for instance \cite{LL}) that a Borel subalgebra $\b$ of a semisimple Lie algebra
 satisfies $\Der(\b)=\Inder(\b)$, $\Z(\b)=0$ and $(\Lambda^2\b)^\b=0$.
\end{ex}

\subsection*{Biderivations in the semisimple case}

Let $(\g,r)$ an (almost) factorizable semisimple 
Lie bialgebra, with $r$ a BD classical r-matrix  \ie $r$ of the
 form as in equation \eqref{BD-r-matrix} of theorem \ref{BD-teo}, that is, for
a fixed non-degenerate, symmetric, invariant, bilinear form on $\g$,
 a certain Cartan subalgebra $\h$, an election of positive and simple roots
$\Phi^+\subset\Phi(\h)$ and $\Delta=\{\alpha_1,\cdots,\alpha_\ell\}$, respectively, a pair of discrete and continuous parameters $(\Gamma_1,\Gamma_2,\tau)$ and $r_0\in\h\ot\h$, respectively, with $r_0=\lambda +\Omega_0$, $\lambda\in\Lambda^2\h$,
$\lambda=\sum\limits_{1\leq i<j\leq \el}\lambda_{ij}h_i\wedge h_j$, where
$h_i:=h_{\alpha_i}$,
the antisymmetric component $r_\Lambda$ of such an r-matrix is of the form
\begin{eqnarray}
\label{rlambda}
r_\Lambda=\sum_{\alpha\in\Phi^+}x_{-\alpha}\w x_\alpha
+\sum_{\alpha\prec\beta}x_{-\alpha}\w x_\beta
+\lambda
\end{eqnarray}
Here $\ell =\dim\h$ is the rank of the Lie algebra $\g$. 
Notice that $\ad _{x}r=\ad _{x}r_\Lambda $
since the symmetric component of $r$ is $\g$-invariant.
 If we are in the real almost factorizable case $(\g,\delta)$, namely
$(\g\ot_\R\C,\delta\ot_\R\C)$ is factorizable, so there exists
$r\in(\g\ot_\R\C)\ot_\C(\g\ot_\R\C)=(\g\ot_\R\g)\ot_\R\C$
with $\delta\ot_\R\id_\C(x)=\ad_x(r)$ for all $x\in\g\ot_\R\C$;
suppose that $r$ is of the form as above, then necessarily 
$\delta(x)=\ad_x(r)=\ad_x(r_{\Lambda})$, for all $x\in\g$. In particular
$r_{\Lambda}\in\Lambda^2_\R\g$.

The goal of this section is to prove the result given in the following theorem. In fact, we exhibit two different proofs of it, namely, one is the aplication of corollary \ref{bider-semi}, which gives proposition \ref{prop-n+}. 
The second is longer but direct and follows in this section, being proposition \ref{prop-n+} 
 the most subtle part of the proof.

\begin{teo}
\label{teo:main}
Let $(\g,r)$ be an (almost) factorizable semisimple Lie bialgebra, with $r$ as
 in the previous paragraph.  If
$D:\g\to \g$ is a biderivation, then $D=\ad_{H}$ for a (unique)
$H\in\h$ satisfying
\[
\alpha(H)=(\tau\alpha)(H)
\hbox {,\quad for all }\alpha\in\Gamma_1
\]
In particular, if the discrete parameter  is empty,
that any $H\in\h$ determines a biderivation and all biderivations are of this type.
\end{teo}

It is usefull to recall the notion of {\em level} or {\em height} of a root; if $\alpha =\sum_{i=1}^\ell n_i\alpha_i$ define the height of $\alpha$
as the integer
\[
\hh(\alpha)=\sum_{i=1}^\ell n_i
\]
The same definition can be extended for any weight $\mu\in\h^*$; in this more general context, the height is a complex number.
The adjoint representation of $\g$ decomposes in {\em height spaces}, explicitly given by 
$$\g_{(n)}:=\underset{\hh(\alpha)=n}{\bigoplus\limits_{\alpha\in\Phi\cup\{0\}:}
\g_\alpha}$$ where we include $\g_{(0)}=\g_0=\h$, so
$\g=\bigoplus\limits_{n\in\ZZ}\g_{(n)}$.
Analogously, the representation $\Lambda ^2\g$ decomposes in weight spaces $(\Lambda ^2\g)_\mu$,
with weights of integer levels.
For each $\mu$ in the $\ZZ$-span of $\Phi$, we have
$$(\Lambda ^2\g)_\mu=\bigoplus\limits_{\alpha,\beta\in\Phi\cup\{0\}:\ \alpha+\beta=\mu}(\g_\alpha\w\g_\beta)$$
This decomposition can be rearranged
as a decomposition in height spaces
as follows:
\[
\Lambda ^2\g=\bigoplus_{\mu\in\h^*}(\Lambda ^2\g)_\mu
=\bigoplus_{n\in\ZZ}\left(\bigoplus_{\mu\in\h^*:\ \hh(\mu)=n}(\Lambda ^2\g)_\mu
\right)
=\bigoplus_{n\in\ZZ}(\Lambda ^2\g)_{(n)}
\]
where
$(\Lambda ^2\g)_{(n)}:=\bigoplus\limits_{\mu\in\h^*:\ \hh(\mu)=n}(\Lambda ^2\g)_\mu
$ is said the {\em component of height} $n$.
Notice that $\Lambda ^2\h\subset(\Lambda ^2\g)_{(0)}$ but also $x_{-\alpha}\w x_\alpha\in(\Lambda ^2\g)_{(0)}$;
moreover, conditions on the BD-triple (see \cite{BD} and \cite{AJ}) force that
$\alpha\prec\beta$ implies $\hh(\alpha)=\hh(\beta)$ and so,
$x_{-\alpha}\w x_\beta\in(\Lambda ^2\g)_{(0)}$,
then, all the terms which appear in $r_\Lambda$ are in $(\Lambda ^2\g)_{(0)}$.
Hence, 
\[
\fbox{$r_\Lambda\in(\Lambda ^2\g)_{(0)}$}
\]
Also, 
from $[\g_\mu,\Lambda^2\g_\nu]\subseteq \Lambda^2\g_{\mu+\nu}$, it is clear that
$[\g_{(n)},(\Lambda ^2\g)_{(k)}]\subseteq(\Lambda ^2\g)_{(n+k)}$.
This discussion implies the following lemma.
\begin{lema}
\label{lema:height}Let $\g=\n_+\oplus\h\oplus\n_-$ be the triangular descomposition related to $\h$ and a to an election of positive roots; if we write $x=x_++x_\h+x_-$
then
$\ad_ x(r)=0$ if and only if
$\ad _{x_{\pm}}(r)=\ad_{x_{\h}}(r)=0$.
Moreover, if $x=\sum_nx_{(n)}$ with $x_{(n)}\in\g_{(n)}$
corresponding to the decomposition of $\g$ in height spaces, then
$\delta x=\ad_x(r)=0$ if and only if
$\delta x_{(n)}=\ad_{x_{(n)}}(r)=0$ for all $n\in\ZZ$.
\end{lema}

The proof of the theorem relies in a close observation of
the adjoint action, explicitly stated in the following lemma. 
\begin{lema}
\label{lema:calculo}
Let $x_\gamma\in\g_\gamma$ with $\gamma\in \Phi^+$, 
then
\begin{eqnarray*}
\ad _{x_\gamma}r &=&h_\gamma\w x_\gamma+
\sum_{\gamma\neq\alpha\in\Phi^+}\left(c_{\gamma,-\alpha}
 x_{-\alpha +\gamma}\w x_\alpha +
c_{\gamma,\alpha} x_{-\alpha}\w x_{\gamma+\alpha}
\right)
\\
&&+\sum_{\alpha\prec\beta:\gamma\neq\alpha,\beta}\left(
c_{\gamma,\alpha} x_{-\alpha +\gamma}\w x_{\beta}
+
c_{\gamma,\beta} x_{-\alpha}\w x_{\beta +\gamma}
\right)
+\sum_{\beta:\gamma\prec\beta}\left(
 h_{\gamma}\w x_{\beta}
+
c_{\gamma,\beta} x_{-\gamma}\w x_{\beta +\gamma}
\right)
\\
&&+\sum_{\alpha:\alpha\prec\gamma}c_{\gamma,-\alpha}
 x_{-\alpha+\gamma}\w x_{\gamma}
+[x_\gamma,\lambda]
\end{eqnarray*}
where $c_{\gamma,\pm\alpha}\in\C$ are the structure constants
such that $[x_\gamma,x_{\pm\alpha}]=c_{\gamma,\pm\alpha}x_{\gamma\pm\alpha}$.
In addition, if we write
$\lambda=\sum\limits_{1\leq i<j\leq \el}\lambda_{ij}h_i\wedge h_j
=\frac12\sum\limits_{i,j=1}^{\el}\lambda_{ij}h_i\wedge h_j$ with
$\lambda_{ji}=-\lambda_{ij}$
then
\begin{eqnarray*}
[x_\gamma,\lambda]&=&-
\!\!\!
\sum_{1\leq i<j\leq \el}
\!\!\!
\lambda_{ij}(
\gamma(h_i)x_\gamma\wedge h_j
+\gamma(h_j)h_i\wedge x_\gamma)=
\left(
\sum_{1\leq i<j\leq \el}
\!\!\!
\lambda_{ij}
(\gamma(h_i) h_j
-\gamma(h_j)h_i)
\right)\wedge x_\gamma
\\
&=&
\frac12
\sum_{i,j=1}^\ell
\lambda_{ij}
(\gamma(h_i) h_j
-\gamma(h_j)h_i)\w x_\gamma
=
\sum_{i,j=1}^\ell
\lambda_{ij}
\gamma(h_i) h_j
\w x_\gamma
\end{eqnarray*}
\end{lema}
\begin{proof}
Straightforward.
\end{proof}
In order to deal with this formula, we simplify it by considering
the following decomposition of $\Lambda^2\g$ 
induced by the triangular decomposition $\g=\h\oplus(\n_+\oplus\n_-)$, namely
\[
\Lambda^2\g=
\Lambda^2\h\oplus (\h\wedge (\n_+\oplus\n_-))
\oplus
\Lambda^2(\n_+\oplus\n_-)
\]
Define $p:
\Lambda^2\g\to
\Lambda^2\h\oplus(\h\wedge (\n_+\oplus\n_-))$,
the canonical projection associated to the above decomposition. The formula of Lemma \ref{lema:calculo} implies the following:
\[p(\ad _{x_\gamma}r)
=h_\gamma\w \left(
x_\gamma
+\sum_{\beta:\gamma\prec\beta}
  x_{\beta}\right)
+
\left(
\sum_{1\leq i<j\leq \el}
\lambda_{ij}
(\gamma(h_i) h_j
-\gamma(h_j)h_i)
\right)\wedge x_\gamma
\]
It is convenient to introduce the element
$H^\gamma_\lambda:=
h_\gamma+
\sum_{i,j=1}^\ell
\lambda_{ij}
\gamma(h_i) h_j$.
Write $\gamma=\sum_{i=1}^\el n_i\alpha_i$ then
\[
H^\gamma_\lambda=
\sum_{j=1} ^\ell\left(n_j+\sum_{i=1} ^\ell
\lambda_{ij}\gamma(h_i)
 \right)h_j
=\sum_{j=1}  ^\ell\left(n_j+\sum_{i,k}
\lambda_{ij}n_k\alpha_k(h_i)
 \right)h_j
\]
Under this notation we have
\[p(\ad _{x_\gamma}r)
=H^ \gamma_\lambda\wedge x_\gamma
+h_\gamma\wedge\left(
\sum_{\beta:\gamma\prec\beta}
  x_{\beta}\right)
\]
\begin{lema}
For any $\gamma\in\Phi$ and any $\lambda\in\Lambda^2\h$, the element $H^\gamma_\lambda\in\h$ is nonzero.
\end{lema}
\begin{proof}Recall we write $\gamma=\sum_{i=1}^\el n_i\alpha_i$; 
if $H^\gamma_\lambda=0$ then,
in particular, $\alpha(H^\gamma_\lambda)=0$ for all $\alpha\in\h^*$, so
\[0=\sum_{m=1}^\ell\alpha_m(H^\gamma_\lambda)n_m
=\sum_{j,m=1} ^\ell n_m
\alpha_m(h_j)n_j
+
\sum_{j,k,m=1} ^\ell
n_k\alpha_k(h_i)\lambda_{ij} \alpha_m(h_j)n_m
 \]
It is  convenient to use matrix notation.
Let us denote by
 $K$ the matrix with entries  $\kappa_{ij}= \alpha_i(h_j)$, 
 $\Lambda$ the matrix
with coefficients $\lambda_{ij}$ and $\sub n=(n_1,\dots,n_\el)$. 
Notice that $K$ is the matrix of the Killing form
restricted to $\h$. The formula above  can be written as
\[
0=\sub n\cdot K\cdot \sub n^t+
\sub n\cdot K\cdot \Lambda\cdot K\cdot \sub n^t
\]
The second term is easily seen to be zero since
\[
\sub n\cdot K\cdot \Lambda\cdot K\cdot \sub n^t=
\left(\sub n\cdot K\cdot \Lambda\cdot K\cdot \sub n^t\right)^t=
\sub n\cdot K^t\cdot \Lambda^t\cdot K^t\cdot \sub n^t=
-\sub n\cdot K\cdot \Lambda\cdot K\cdot \sub n^t
\]
where the first equality holds because we are transposing a complex number, the second is valid
for any product of matrices, and the last
uses the fact that $K$ is symmetric and $\Lambda$ anti-symmetric.
Besides, in the basis $\{h_1,\dots,h_\ell\}$, the matrix $K$ is {\em real
symmetric and positive defined}
(see for instance \cite{Kn}, corollary 2.38), hence
\[\sum_{m}\alpha_m(H)n_m
 =\sum_{j,m}n_m\kappa_{mj}n_j=
\sub n \cdot K\cdot \sub n^t >0\
\forall \sub n\in\R^n\setminus 0;\]
in particular, it gives a nonzero real number for any $0\neq (n_1,\dots,n_\el)\in
\ZZ ^n$.
\end{proof}

\begin{prop}\label{prop-n+}
Let $x\in\n_+\oplus\n_-$, then $\ad_x(r)=0\iff x=0$.
\end{prop}
\begin{proof}
Let $x=\sum_{\gamma\in\Phi}c_\gamma x_\gamma$ with $c_\gamma$ arbitrary, and suppose
$\ad_xr=0$. Since $\ad_{-}(r)$ preserves the height 
(see Lemma \ref{lema:height}), we can consider
different heights separately. Since we will not need such refinement
in all its strengh, we
will only consider separately the cases $\hh(\gamma)>0$ or
$\hh( \gamma)<0$, namely $\gamma$ a positive or negative root.

So let us consider an element
$x=\sum_{\gamma\in\Phi^+}c_\gamma x_\gamma$, 
the case in $\Phi  ^-$ is analogous. We have 
\[
0=
p(\ad_xr)=
\sum_{\gamma\in\Phi^+}
c_\gamma 
H^\gamma_\lambda\w  x_\gamma
+
\sum_{\gamma\in\Phi^+}
c_\gamma h_\gamma\wedge
\left(\sum_{\beta:\gamma\prec\beta}
  x_{\beta}\right)
\]
Denote $\wt\Gamma$ the $\ZZ$-span of the
discrete parameter;
we claim that if $\gamma\in\wt \Gamma\cap \Phi^+$, then $c_{\gamma}=0$. To see this, consider
 $\Gamma_0\subset\Phi^+$  the set of minimal elements
 $\gamma\in\wt\Gamma$
such that $c_\gamma\neq 0$ (minimal with respect  to $\prec$). 
Notice that the terms with $x_\beta$ where 
$\gamma\prec \beta$ can not cancel any term with $x_\gamma$ for 
$\gamma\in\Gamma_0$, because if  $c_\beta\neq 0$ then
$\gamma$ could not  be minimal.
Hence, if we consider only elements in
$\Gamma_0$, neccesarily
\[
0=
\sum_{\gamma\in\Gamma_0}
c_\gamma H^\gamma_\lambda\wedge x_\gamma
\]
since the $\{x_\gamma\}_{\gamma\in\Gamma_0}$ are linearly independent, then
\[
0=
c_\gamma H^\gamma_\lambda
\quad \forall \gamma\in\Gamma_0
\]
But $H^\gamma\neq 0$ implies
 $c_{\gamma}=0$, which is absurd because
$\gamma\in\Gamma_0$. We conclude 
that $c_\gamma=0$ for all $\gamma$ in  $\wt\Gamma$. Hence
the equality $\ad_x(r)=0$ implies
\[
0=
p(\ad_xr)=
\sum_{\gamma\in\Phi^+}
c_\gamma 
H^\gamma_\lambda\w 
x_\gamma
\]
Now we can repeat word by word the same argument as in case $\gamma\in \Gamma_0$,
namely the linear independence of the $x_\gamma$ implies $c_\gamma H^\gamma_\lambda
=0$ for all $\gamma\in\Phi^+$, but $H^\gamma_\lambda\neq 0$ implies $c_\gamma=0$ 
for all $\gamma\in \Phi^+$.
\end{proof}
\begin{proof}[Proof of theorem \ref{teo:main}]
In order to conclude the proof, 
almost all work is done.
 We know that if $\ad_x(r)=0$ then $x\in \g_{(0)}=\g_0=\h$.
Notice that the standard component and the continuous parameter have
total weight equal to zero, \ie $\ad_H(r_{st}+\lambda)=0$
 for $H\in \h$, then the only terms surviving in $\ad_H(r)$ are
\[
\ad_H(r)=\sum_{\alpha\prec\beta}\ad_H(x_{-\alpha}\wedge x_{\beta})=
\sum_{\alpha\prec\beta}(\beta(H)-\alpha(H))x_{-\alpha}\wedge x_{\beta}
\]
so $\alpha(H)=\beta(H)$ for all $\alpha\prec\beta$, and that is equivalent
to 
$\alpha(H)=(\tau\alpha)(H)$ for all $\alpha\in \Gamma_1$.
At this stage, we have finished the description of $\Ker(\delta)$, but
in virtue of Corollary \ref{coro:bider}, this 
implies as well a description of the biderivations in $(\g,\delta)$. 
Notice that in the real case, even if $r\in(\g\ot_R\C)\ot_\C(\g\ot_R\C)
\setminus \g\ot_\R\g$, we know that $r_{\Lambda}\in\Lambda^2_\R\g$,
so the proof of the complexified Lie algebra descends to the real form $\g$. 
\end{proof}


\begin{rem}
Corollary \ref{bider-semi} says that if $\ad_x\in\Bider(\g,r)$ then
$x\in\h$, so corollary \ref{bider-semi} together with the very last argument above
gives an alternative proof of 
theorem \ref{teo:main}.
\end{rem}

\subsection*{Extension of scalars}

For a given Lie bialgebra, it is possible to define a (double) complex 
of the form $C^{p,q}\g=\Lambda^p\g^*\ot\Lambda^q\g$, where the vertical
differentials are the Chevalley-Eilenberg differential of $\g$ with coefficients
in $\Lambda^q\g$, and horizontal  differentials are the dual of the Chevalley-Eilenberg
differential correspondig to the Lie coalgebra structures.
This complex
was first describe in
\cite{LR}. In particular,
for $p=q=1$, if one identifies $\g^*\ot \g=\Hom(\g,\g)=\End(\g)$, we get that
the kernel of the vertical differential consists on derivations (the image of the
preceding differential are the inner ones), and the kernel of horizontal differential
consists  on coderivations, so the kernel of both differentails is precisely
the set of biderivations. As a consequence, the set of biderivations extends
scalars in the sense that if $\K\subset \E$ is a field extension, then
$\bider_\E(\g\ot_\K\E)=\bider_\k(\g)\ot_\K\E$, and a given biderivation $D$
of a $\K$-Lie algebra $\g$ is inner if and only if $D\ot_\K\Id_\E$ is
inner as biderivation of $\g\ot_\K\E$.

\section{Main Construction for Trivial Abelian Extensions}
Along this section, we denote by $V$ a $d$-dimensional vector space over a field $\K$, $\{t_1,\dots,t_d\}$ a basis of $V$ and 
$\{t_1^*,\dots,t_d^*\}$ the associated dual basis of $V^*$.

\begin{teo}
\label{propmachine}
Let $(\g,\delta_\g)$ be a Lie bialgebra, 
$(V,\delta_V)$  a $d$-dimensional Lie coalgebra, $V^*$ the dual Lie algebra and $\DD:V^*\to\bider( \g)$
a Lie algebra map, then the following map defines a Lie bialgebra structure on 
$\l=\g\times V$, for all $x\in\g$ and $v\in V$:
\[
\delta(x+v)=\delta_\g(x)+2 \sum _{i=1}^dD_i(x)\w t_i+\delta_V(v)
\]
where $\{t_1,\dots,t_d\}$ is a basis of $V$,
$\{t_1^*,\dots,t_d^*\}$ the dual basis of $V^*$ and
 $D_i=\DD(t_i^*)$, $1\leq i\leq d$.
\end{teo}

\begin{proof}
We  need to prove co-Jacobi and the 1-cocycle condition.
In order to prove co-Jacobi for $\delta$, 
for any linear function $f:\g\to\Lambda ^2(\g)$, denote by $\partial _f \colon\Lambda ^2(\g)\to\Lambda ^3(\g)$ the map
given by $\partial _f(x\wedge y)=f(x)\wedge y-x\wedge f(y)$. 
So, under this notation, $\delta$ satisfies co-Jacobi if and only if $\partial_{\delta}\circ\delta=0$.
Notice that $\partial _{f+g}=\partial _f+\partial _g$, so
$$\partial_{\delta}=\partial_{\delta_\g}+2\sum_{i=1}^d\partial_{D_i(-)\w t_i}
+\partial_{\delta_V}
$$

Let us prove  first that $\partial_{\delta}\circ\delta (x)=0$ for any $x\in\g$. 
\[
\begin{array}{rcccccc}
\partial_{\delta}(\delta(x))&=&\partial_{\delta}(\delta _{\g}(x))&&&+&2\sum\limits_{i=1}^d\partial_{\delta}(D_ix\wedge t)
\\
&=&\partial_{\delta _{\g}}(\delta _{\g}(x))&+&2\sum\limits_{i=1}^d\partial_{D_i\wedge t_i}(\delta _{\g}(x))
&+&2\sum\limits_{i=1}^d(\delta(D_ix)\wedge t_i-D_ix\wedge \delta (t_i))
\\
&=&A&+&B&+&C\\
\end{array}\]
where these three terms are computed separetely as follows. The first term,
A=$\partial_{\delta _{\g}}(\delta _{\g}(x))$
has to be zero  since
$\delta _{\g}$ satisfies co-Jacobi.
For the second term,
\[\frac12B=
\sum_{i=1}^d\partial_{D_i\wedge t_i}(\delta _{\g}(x))
=\sum_{i=1}^d(D_ix_1^{\g}\wedge t_i\wedge x_2^{\g}-x_1^{\g}\wedge D_i x_2^{\g}\wedge t_i)
\]
\[
=-\sum_{i=1}^d\left((D_i\ot \id +\id\ot D_i)
(x_1^{\g}\wedge x_2^{\g})\wedge t_i\right)
\]
\[
=-\sum_{i=1}^d\left(
(D_i\ot \id +\id\ot D_i)
(\delta _\g x)\wedge t_i
\right)
\in \Lambda^2\g\w V
\]
where we used the Sweedler type notation
$\delta_\g x=x_1^\g\w
x_2^\g$. Half of the third term equals
\[\frac12C=\sum_{i=1}^d(\delta(D_ix)\wedge t_i-D_ix\wedge \delta (t_i))
\]\[
=\sum_{i=1}^d\delta_{\g}(D_ix)\wedge t_i
+2\sum_{i,j=1}^dD_j(D_ix)\wedge t_j\wedge t_i
-\sum_{i=1}^dD_ix\wedge \delta _V(t_i)
\]
Notice that, in C, only the first sum belongs to $\Lambda^2\g\w V$, and it cancels with B
becauuse $D_i$ are coderivations.
It only remains to verify that  the second and third terms of 
C cancel each other, or, equivalently, that the following identity holds
\begin{eqnarray}2
\sum_{i,j=1}^dD_j(D_ix)\wedge t_j\wedge t_i
=\sum_{i=1}^dD_ix\wedge \delta _V(t_i)\label{vienen-de-C1}
\end{eqnarray}
 Observe that in the left hand side we have
$$2\sum_{i,j=1}^dD_j(D_ix)\wedge t_j\wedge t_i
=\sum_{i,j=1}^d[D_j,D_i](x)\wedge t_j\wedge t_i
$$
because $ t_j\wedge t_i$ is antisymmetric in the indexes $i,j$. 
On the right hand side of (\ref{vienen-de-C1}), we may write $\delta _V(t_k)$ as a linear combination of the $ t_j\wedge t_i$, explicitly
$
\delta _V(t_k)=
\sum_{i,j=1}^dc_k^{j,i} t_j\wedge t_i
$.
So,  identity (\ref{vienen-de-C1}) is also equivalent to
\[[D_j,D_i]=
\sum_{k=1}^dc_k^{j,i}D_k
\]
which holds because the map
$\DD:V^*\to \Bider(\g)$, $t_i^*\mapsto D_i$,  is a Lie algebra map.

Finally, 
$\delta|_V=\delta_V$ and $ \delta_V(V)\subseteq\Lambda^2V$ since by construction
$(V,\delta_V)$ is a Lie subcoalgebra. Hence,
$\partial_{\delta}\delta(v)=\partial_{\delta}\delta_V(v)=\partial_{\delta_V}\delta_V(v)=0
$ for any $v\in V$.

\end{proof}

\begin{ex} As a toy example, consider $\g=\aff_2(\K)$ the non-abelian 2-dimensional Lie algebra, whith basis
$\{h,x\}$ and bracket $[h,x]=x$, and $V=\K t$. All possible cobrackets in $\aff_2(\K)$
up to isomorphism of Lie bialgebras are 
(see \cite{FJ}) as follows:
\begin{enumerate}
\item $ \delta^0(h)=h\w x$, $\delta^0(x)=0$, in this case 
$\D=-\ad_x$ and $\bider(\aff_2,\delta^0)=\K\ad_x$;  or
\item 
the 1-parameter family
$ \delta_\mu(h)=0$, 
$\delta_ \mu(x)=\mu h\w x$,  $\mu\in\K$, so  $\D=\mu\ad_h$. In this case,
 $\bider(\aff_2,\delta_\mu)=\K\ad_h$ if $\mu\neq 0$
and $\bider(\aff_2,\delta_\mu)=\Der(\aff_2)$ if $\mu=0$.
Notice that $\Der(\aff_2)=\InDer(\aff_2)$.
\end{enumerate}
The biderivations given above were easily obtained by means of corollary \ref{coro:bider}.
The procedure described in  theorem 
\ref{propmachine} says that 
if $D\in\bider(\aff_2(\K),\delta_{\aff_2})$,
\[
\delta(t)=0,\quad 
\delta(u)=\delta_{\aff_2}(u)+D(u)\w t, \quad  \forall u\in\aff_2(\K)
\]
is a Lie cobracket on $\aff_2(\K)\times \K$. 
We obtain the whole list of possible such choices:
\begin{center}
{\em Possible Lie cobrackets on $\aff_2(\K)\times \K$}
\end{center}
\[
\begin{array}{|rllcc|}
\hline
i)&\delta(x)=0&\delta(h)=h\w x+x\w t& \delta(t)=0&\\
\hline
ii)&\delta(x)=0&\delta(h)=h\w x& \delta(t)=0&\\
\hline
iii)&\delta(x)=\mu h\w x +x\w t&\delta(h)=0&\delta(t)=0 &\mu\neq 0\\
\hline
iv)&\delta(x)=\mu h\w x&\delta(h)=0&\delta(t)=0 & \mu\neq 0\\
\hline
v)&\delta(x)=D(x)\w t&\delta(h)=D(h)\w t, \ D\in\Der(\aff_2)&\delta(t)=0&\\
\hline
\end{array}
\]
Notice that the Lie bialgebra of the case  $i)$ is isomorphic
to the one of case $ii)$ by means of the map
$x\mapsto x;\ h\mapsto h + t
;\ t\mapsto t$.
Analogously, the Lie bialgebra of the case  $iii)$ with parameter $\mu$ is isomorphic
to the one of case  $iv)$ with parameter $-\mu$, by means of the map
$x\mapsto x;\ h\mapsto h + \frac1\mu t
;\ t\mapsto t$. In case $v)$ if the derivation $D=\ad_{\alpha x+\beta h}$ then
$D:x\mapsto \beta x,\ h\mapsto -\alpha x$,
so the cobracket has the form
$\delta x=\beta x\w t$,  $\delta t=0$ and $\delta h=-\alpha x\w t$.
In matrix notation, choosing basis $\{x,t,h\}$ of the Lie algebra
 $\l=\aff_2\times \K t$, and
basis $\{x\w t,t\w h,h\w x \}$ of
$\Lambda^2\l$, the  cobrackets given by the above construction are
\[ii)
\left(
\begin{array}{ccc}
0&0&0\\
0&0&0\\
0&0&1
\end{array}
\right);\
iv)\left(
\begin{array}{ccc}
0&0&0\\
0&0&0\\
\mu&0&0
\end{array}
\right);\
v)\left(
\begin{array}{ccc}
\beta&0&-\alpha\\
0&0&0\\
0&0&0
\end{array}
\right)
\]
Case $v)$
is isomorphic to
$\left(
\begin{array}{ccc}
0&0&1\\
0&0&0\\
0&0&0
\end{array}
\right)
$ 
if $\alpha\neq 0$, $\beta=0$, 
simply by considering the transformation $x\leftrightarrow \frac1\alpha x$. If
$\beta\neq 0$, then $v)$ is isomorphic to
$\left(
\begin{array}{ccc}
1&0&0\\
0&0&0\\
0&0&0
\end{array}
\right)
$ by the transformation $x\mapsto x$, $h\mapsto h+\frac\alpha\beta x$, $t\mapsto 
\beta t$.
If one compares all this possibilities with the classification result in \cite{FJ}
for the Lie algebra $\r_{3,\lambda=0}$ one sees that we have covered all possibilities.
This is not surprising due to the following result. 
\end{ex}

Next theorem says that with some extra hypothesis,  proposition \ref{propmachine} has its converse. See
 table of example \ref{tabla}
 for non-semisimple examples  where next theorem applies.

\begin{teo}
\label{teo-producto}
Let $\g$ be a Lie algebra such that $(\Lambda^2\g)^\g=0$ and
 $\Z(\g)=0$; let $V$ be a vector space considered as abelian Lie algebra;
assume that either  $\dim V>1$ and $[\g,\g]=\g$, or $\dim V=1$.
If $\l=\g\times V$ 
then 
all Lie cobrackets on $\l$ are as in theorem
\ref{propmachine}.
 Explicitly, 
if $\delta$ defines a Lie bialgebra structure on $\l$, then
\begin{enumerate}
\item  $\delta(V)\subseteq\Lambda^2(V)$,
so, $V$ is a Lie subcoalgebra with $\delta_V=\delta|_V$.
In particular, $V$ is an ideal and a coideal,
hence, $\l/V$ inherits a unique Lie bialgebra structure such that
$\pi:\l\to \l/V$ is a Lie bialgebra map.
\item Let $\pi_\g:\l\to\g$ be the canonical projection associated to the decomposition $\l=\g\times V$, then $\delta_{\g}:=(\pi _{\g}\w\pi _{\g})\circ\delta |_{\g}\colon \g\to \Lambda^2\g$
is a Lie bialgebra structure on $\g$ and $\l /V\cong\g$ canonically as Lie bialgebras.
\item
$\delta(\g)\subseteq \Lambda^2\g\oplus \g\wedge V$.
 If $\{t_i\}_{i=1}^d$ is a basis of $V$, then
for any $x\in\g$, $\delta(x)$ is of the form
\[
\delta x=\delta _\g x+2\sum_{i=1}^dD_ix\w t_i
\]
 where $D_i:\g\to\g, i=1,\dots,d$, are derivations and coderivations of $(\g,\delta _\g)$.
 The linear subspace generated by $\{D_1,\cdots,D_d\}$ is a Lie subalgebra of $\BiDer(\g)$; moreover,
the map $\DD\colon (V^*,\delta_V^*)\to \BiDer(\g)$ defined by $\DD(t_i^*)=D_i$, 
is a Lie algebra map.

\item 
Let $(\l,\delta )$ be the Lie bialgebra  associated to a data 
$(\l,\delta_\g,\delta_V,D_1,\cdots ,D_d)$. Let  $\Phi =(\phi_\g,\phi_V)$ be a linear automorphism of $\l$, with $\phi_\g$  a Lie algebra automorphism of $\g$ and $\phi_V\in\GL(V)$.
 If we denote by $\wt \delta_\g=(\phi_\g\wedge\phi_\g)\circ \delta_\g\circ\phi_\g^{-1}$, $\wt\delta_V=(\phi_V\wedge\phi_V)\circ \delta_V\circ\phi_V^{-1}$,
$\wt D_i:=\sum _{j=1}^dA_{ij}\phi_\g\circ D_j\circ \phi _\g^{-1}$, where
 $\phi_V(t_j)=\sum _{i=1}^dA_{ij}t_i$, $1\leq j\leq d$,
and $(\l,\wt\delta)$ the Lie bialgebra  associated to the data $(\l,\wt\delta_\g,\wt\delta_V,\wt D_1,\cdots ,\wt D_d)$, then
$\Phi:(\l, \delta)\to^{\!\!\!\!\!\!\!\sim} (\l,\wt \delta)$ is a Lie bialgebra isomorphism. 
If, moreover, $[\g,\g]=\g$ then any Lie bialgebra isomorphism from 
$(\l, \delta)$ to $(\l,\wt \delta)$ is of this form.
\end{enumerate}
\end{teo}

\begin{proof}
Consider the decomposition
$\Lambda^2(\l)=\Lambda^2(\g)\oplus \g\w V \oplus\Lambda^2(V)$.
It is straightforward to see that if $\delta(\g)\subseteq\Lambda^2\g\oplus\g\w V$
and $\delta(V)\subseteq\Lambda^2 V$
then the implications in the proof of the theorem \ref{propmachine} can be reversed. This
will prove  items 1, 2, 3. So, let us see 
 $\delta(V)\subseteq\Lambda^2 V$ first:

By proposition \ref{central} together with $\Z(\g)=0$, we have
\[
\delta (V)= \delta (\Z(\l))\subseteq (\Lambda ^2\l)^\l = 
(\Lambda ^2\g)^\g\oplus \Z(\g)\w V\oplus \Lambda ^2 V= \Lambda ^2 V
\]
On the other hand,  $\delta(\g)\subseteq\Lambda^2\g\oplus\g\w V$ is trivial in case dim $V=1$ since in this case, $\Lambda ^2 V= 0$. 
If $\dim V>1$,  assuming $\g =[\g,\g]$ then also $\g =[\l,\l]$, so by proposition  \ref{central}  
\[
\delta(\g)=\delta([\g,\g])=\delta([\l,\l])\subset [\l,\l]\w \l=
\g\w \l =\Lambda^2\g\oplus\g\ot V
\]

4. Notice that $\wt\delta_\l=(\Phi\wedge\Phi)\circ \delta_\l\circ\Phi^{-1}$ if and only if
$\wt\delta_\g=(\phi_\g\wedge\phi_\g)\circ \delta_\g\circ\phi_\g^{-1}$,  $\wt\delta_V=(\phi_V\wedge\phi_V)\circ \delta_V\circ\phi_V^{-1}$ 
 and
\[
(\Phi\wedge\Phi)\left(\sum _i
D_i(\phi_\g^{-1}x)\wedge t_i\right)
=\sum_i\wt D_i(x)\wedge t_i
\]
The identities concerning $\wt\delta_\g$ and $\wt\delta_V$ are true by hypothesis. For the last, notice that 
\[
(\Phi\wedge\Phi)\left(D_i(\phi_\g^{-1}x)\wedge t_i\right)
=\left(\phi_\g D_i \phi_\g^{-1}(x)\right)\wedge \phi_V(t_i);
\]
write 
$ \phi_V(t_i)=\sum _jA_{ij}t_j$, then
\[
(\Phi\wedge\Phi)\left(D_i(\phi_\g^{-1}x)\wedge t_i\right)
=\sum _jA_{ij}\phi_\g\left(D_i(\phi_\g^{-1}x)\right)\wedge t_j
\]
For the converse, if $(\l,\delta_\l)$ and $ (\l,\wt\delta_\l)$ are Lie bialgebras, then we have the corresponding  
$(\delta_\g, \delta_V,  D_1,\cdots, D_d)$ and 
$(\wt\delta_\g, \wt\delta_V, \wt D_1,\cdots,\wt D_d)$. If they are isomorphic Lie bialgebras, then there exists a Lie algebra isomorphism $\Phi:\l \to \l$  such that 
$\wt\delta_\l=(\Phi\wedge\Phi)\circ \delta_\l\circ\Phi^{-1}$. It is 
necessary to prove that it induces the existence of 
$\phi_\g:\g\to\g$ and $\phi_V:V\to V$, or, in other words,
that $\Phi(\g)\subseteq \g$ and $\Phi(V)\subseteq V$.
This holds because $[\g,\g]=\g$ and $\Z(\g)=0$
imply
$\Phi(\g)=\Phi([\g,\g])=\Phi([\l,\l])=[\l,\l]= [\g,\g]=\g$ and 
$\Phi(V)=\Phi(\Z(\l))=\Z(\l)=V$.

\end{proof}

Specializing the main theorem to the case of dim $V=1$, we obtain the following.

\begin{coro}Let $\l=\g\times \K$ be a Lie bialgebra where the underlying Lie algebra is 
the product of the Lie algebra $\g$ and the field $\K=\langle t\rangle$ considered as trivial one-dimensional Lie algebra; 
suppose that $\Z (\g)=0$ and $(\Lambda^2\g)^{\g}=0$; 
then
the Lie bialgebra structures on $\l$ are determined by pairs 
$(\delta_\g,D)$, where $\delta_\g$ is a Lie bialgebra structure on $\g$ and
$D\in\BiDer(\g,\delta_\g)$.
The Lie cobracket on $\l$ is explicitly given by 
$\delta (x)=\delta_{\g}(x)+D(x)\wedge t$,
 for any $x\in\g$, and $\delta (t)=0$.
\end{coro}

\begin{ex}
\label{tabla}
The following table exibits some properties of the non abelian, real 3-dimensional  Lie algebras.
We see that there are non semisimple examples of $\g$
where $\Z (\g)=0$ and $(\Lambda^2\g)^{\g}=0$;
so, the previous result applies
in order to describe Lie bialgebra structures on 4-dimensional real Lie
algebras of type
 $\g\times \R$.
\[
\begin{array}{c}
\hbox{\em Invariants of 3-dimensional real Lie algebras}\\
\begin{array}{||c|c|c|c||}
\hline
\g&\Z\g&(\Lambda^2\g)^\g&[\g,\g]\\
&&&\\
\hline
\h_3 : [x,y]= z&\R z&\R x\wedge z\oplus \R y\wedge z&\R z\\
\hline
\r_3 &&&\\
{}[h,x]=x,\ [h,y]=x+y,\ [x,y]=0&0&0&\R x\oplus \R y\\
\hline
\r_{3,\lambda} && &\\
{}[h,x]=x,\ [h,y]=\lambda y,\ [x,y]=0:&&&\\
\hdashline[0.5pt/5pt]
\lambda\in(-1,1],\lambda\neq 0&0&0&\R x\oplus \R y\\
\hdashline[0.5pt/5pt]
 \lambda=-1&0&\R x\wedge y&\R x\oplus \R y\\
\hdashline[0.5pt/5pt]
 \lambda=0&\R y&0&\R x\\
\hline
\r'_{3,\lambda}, \ \lambda\geq 0&&&\\
{}[h,x]=\lambda x-y,[h,y]=x+\lambda ,[x,y]=0&0&0&\R x\oplus \R y\\
\hline
\su(2)&0&0&\su(2)\\
\hline
\sl(2,\R)&0&0&\sl(2,\R)\\
\hline
\end{array}
\end{array}
\]
\end{ex}

The hypothesis   $\Z (\g)=0$ and $(\Lambda^2\g)^{\g}=0$  hold in the semisimple case.

\begin{coro}
\label{coross1}
Let $\g$ be a semisimple Lie algebra, (so all cocycles on $\g$
are coboundary and every derivation on $\g$ is inner), then
\begin{enumerate}
\item All Lie bialgebra structures on $\l=\g\times\K t$ are coboundary and determined by  a Lie bialgebra structure on $\g$, denoted by
$\delta_\g(x)=\ad_x(r)$, with $r\in\Lambda^2\g$ satisfying
$[r,r]\in(\Lambda^3\g)^\g$, and a biderivation $D:(\g,\delta_\g)\to(\g,\delta_\g)$, which is necessarily of the form
$\ad_H$ with $H\in\Ker\delta_\g$. 
The cobracket on $\l$
is given by
\[\delta (x)=\ad_x(r)+[H,x]\wedge t=\ad_x(r-H\w t)\]
 for any $x\in\g$ and $\delta (t)=0$.
Since we may choose $t$ up to scalar multiple,  the element
$H$ may be modify by a nonzero scalar without changing
the isomorphism class of the Lie bialgebra.

\item
Assume in addition that $(\g,\delta_\g)$ is (almost)
 factorizable,  $\delta_\g(x)=\ad_x(r)$ with $r$
 given by a BD-data, i.e. a Cartan subalgebra $\h$, simple
roots $\Delta$, a BD-triple $(\Gamma_1,\Gamma_2,\tau)$
and a continuous parameter with skewsymmetric component $\lambda\in\Lambda ^2\h$,
with $r_\Lambda$ as in equation
\eqref{rlambda}.
Then  $H\in\Ker \delta_g$ if and only if
 $H\in\h$ and $\tau\alpha (H)=\alpha (H)$ for all $\alpha\in\Gamma_1$. 
\end{enumerate}\end{coro}

\begin{ex} {\em Lie bialgebra structures on $\gl(2,\R)=\sl(2,\R)\times \R t$}.
Let $\delta$ be any Lie bialgebra structure on $\sl(2,\R)$, 
which is a simple Lie algebra, 
$\delta=\partial r$.
From \cite{FJ}, we know that there are factorizable, almost factorizable
and triangular structures on $\sl(2,\R)$.
Let $\{x,h,y\}$ be the usual basis of $\sl(2,\R)$.
\begin{enumerate}
\item[
{\em Case 1.}] If $r=h\w x$, then $H_r=2x$ (which is not a
regular,  but a nilpotent element).
We get that 
$ah+bx+cy$ commutes with $x$ if and only if $a=c=0$,
so 
$\bider(\sl(2,\R))=\R\ad_x$.
In particular, $(\sl(2,\R),r)$ is a triangular Lie bialgebra.  
\item[
{\em Case 2.}] If $r=x\w y$ then $H_r=h$, then every biderivation
is a multiple of $\ad_h$ in this case.
 In particular, $(\sl(2,\R),r)$ is a factorizable Lie bialgebra.  
\item[
{\em Case 3.}] If $r=h\w (x+y)$ then $H_r=x-y$ is semisimple non-diagonalizable. One can easily 
check that  every biderivation is a multiple of $\ad_{x-y}$. In particular, $(\sl(2,\R),r)$ is an almost factorizable (non factorizable) Lie bialgebra.  
\end{enumerate}
Hence, we obtain the following description.

\begin{coro}
An exhaustive list of  isomorphism classes of Lie bialgebra structures on $\gl(2,\R)=\sl(2,\R)\times \R t$ is given as follows.
\begin{itemize}
\item[a)] With non-zero cobracket on $\sl(2,\R)$:
\begin{enumerate}
\item Let  
$r=\pm h\w x$ if $D=0$, or $
r=\pm  (h\w x+x\w t)$ if $D\neq 0$; in particular, $(\gl(2,\R),r)$ is a triangular Lie bialgebra. 
\item Let  $r=\beta x\w y$ if $D=0$,  or $
r=\beta(x\w y+h\w t)$ if $D\neq 0$, $\beta\in\R_+$;  in particular, $(\gl(2,\R),r)$ is a factorizable Lie bialgebra. 
\item Let $r=\alpha h\w (x+y)$ if $D=0$,  or $
r=\alpha (h\w (x+y)+(x-y)\w t)$ if $D\neq 0$, $\alpha\in\R\setminus \{0\}$;  in particular, $(\gl(2,\R),r)$ is an almost factorizable Lie bialgebra. 
\end{enumerate}

\item[b)] With zero cobracket on $\sl(2,\R)$ we have $\bider(\sl(2,\R))=\Der(\g)$, then
there are three nontrivial isomorphism classes and, in each ot them, there is a unique derivation up to the action of $\SL(2,\R)$.
In each case, $(\gl(2,\R),r)$ is a triangular Lie bialgebra.
\begin{enumerate}

\item If $D=\ad_{\wt X}$ with $\wt X$ nonzero nilpotent, conjugated to $x$,
then $r=x\w t$.
\item If $D=\ad_{\wt H}$ with $\wt H$ semisimple diagonalizable, conjugated to any non-negative  multiple of $h$, then $r=h\w t$ or zero.
\item If $D=\ad_{\wt U}$ with $\wt U$ semisimple non-diagonalizable, conjugated to any multiple of $x-y$, then $r=(x-y)\w t$.
\end{enumerate}
\end{itemize}
\end{coro}
\end{ex}

\begin{obs}If  $\g$ is a simple Lie algebra, then any Lie bialgebra structure on $\g$ is either triangular or (almost) factorizable. 
If $\g$ is a semisimple Lie algebra, then any Lie bialgebra structure on it may be triangular, (almost) factorizable or none on them, depending on the situation in each component.
\end{obs}

\begin{obs} {\em Dimension of the space of solutions.}
Let $\l=\g\times\K\cdot t$, with $\g$ a semisimple Lie algebra, and suppose that $(\g,\delta_\g)$ 
is (almost) factorizable such that $\delta_\g=\partial r_{\Lambda}$, with $r_{\Lambda}$ of BD-form, then
 the corollary \ref{coross1} applies. 
For each BD triple $(\Gamma_1,\Gamma_2,\tau)$,
\cite{BD} gives the dimension of the space of solutions of the skewsymmetric component of all possible continuous parameters $\lambda\in\Lambda ^2\h$, namely $\dfrac {k(k-1)}{2}$ if $k=|\Delta \setminus \Gamma _1|$. Besides, there are  $|\Gamma _1|$ amount of equations for
 the possible  $H\in\h$ such that $\tau\alpha (H)=\alpha (H)$ for all $\alpha\in\Gamma_1$; this
gives in addition $|\Delta\setminus \Gamma _1|$; 
hence, the set of pairs $(\lambda ,H)$ is an affine space of dimension
$\dfrac {k(k-1)}{2}+k=\dfrac {k(k+1)}{2}$ for each  BD triple. Since we may choose $t$ up to scalar multiple, this dimension is,
indeed, one unit less.
\end{obs}

\begin{ex}The r-matrices corresponding to all the (almost) factorizable Lie bialgebra structures on real forms of complex simple Lie algebras  are given in \cite{AJ}. This, together with the techniques explained in this section, gives an exhaustive list of Lie bialgebra isomorphism classes on real Lie algebras of the form $\s\times V$, with $\s$ a real form of a complex simple  Lie algebra with a given (almost) factorizable structure. For instance, $\u(n)=\su(n)\times \R$, 
 $\u(p,q)=\su(p,q)\times \R$.
\end{ex}

\begin{ex}The classification of 
 the Lie bialgebra structures on three dimensional real  Lie algebras, 
both in the (almost) factorizable and in the triangular case,  is given in \cite{FJ}. This, combined with the results of this section, provides all the Lie bialgebra isomorphism classes on real Lie algebras of shape $\g\times V$, with $\g$ any  three dimensional real  Lie algebra such that $\Z\g=0$ and $(\Lambda ^2\g)^ \g=0$. For instance,
$\su(2)$, $\sl(2, \R)$, 
$\r_3$, $\r_{3,\lambda} $, with $0\neq \lambda\in(-1,1]$,
and $\r'_{3,\lambda}$ satisfy the hypothesis, 
so  the theorem \ref{teo-producto} applies for $\l=\g\times \R$.
However, among them there may be some repetitions, since in general we do not have
$[\g,\g]=\g$ if $\g$ is solvable.
\end{ex}

\subsection*{Abelian extensions with $\dim V>1$}
If $\dim V>1$, there are more possibilities  than $\DD=0$ or $\DD\neq 0$; we can stratify them
by the dimension of the image of $\DD$. 
If the image of a linear map $\DD:V^*\to \Bider(\g)$ is $d_0$-dimensional, $0\leq d_0\leq d$, consider
a basis  $\{t_1,\dots,t_d\}$ of $V$ and the corresponding
dual basis  $\{t_1^*,\dots,t_d^*\}$ of $V^*$ such that
$\{t_{d_0+1}^*,\dots,t_d^*\}$ is a basis of $\Ker\DD$, namely,
$D_1,\dots,D_{d_0}$ are linearly independent and $D_{d_0+1}=\cdots=D_d=0$.
The condition $[D_i,D_j]=\sum_{k=1}^dc_k^{ij}D_k
=\sum_{k=1}^{d_0}c_k^{ij}D_k$ determines uniquely
 $c_k^{ij}$ for $k=1,\dots ,d_0$ in terms of the constant structures
of the Lie algebra Im$(\DD)\subseteq \Bider$. In the case $(\g,\delta_\g)$  semisimple and
 factorizable, we know that $\Bider(\g,\delta_\g)\subseteq\h$ (theorem \ref{teo:main})
, which is abelian, so the general 
theorem \ref{teo-producto}
specializes in the following
result:

\begin{prop}\label{teo-producto2}
Let $\g$ be a  semisimple Lie algebra over $\K$, $V=\K^d$, the abelian
Lie algebra of dimension $d$.
 Consider
$\l=\g\times V$ the trivial abelian extension of the Lie algebra $\g$ by  $V$.
If $\delta:\l\to\Lambda^2\l$ is a Lie bialgebra structure on $\l$
such that $(\g,\delta_{\g})$
is an (almost) factorizable Lie bialgebra,  $\delta_\g(x)=\ad_x(r)$ for all $x\in\g$,
with $r$ given by a BD-data $\h$, $\Delta$, $(\Gamma_1,\Gamma_2,\tau)$, $\lambda\in\Lambda^2 \h$, then,
 there exists a basis \{$t_1,\dots,t_d\}$ of $V$ and
$H_1,\dots H_{d_0}\in \h$ linearly independent elements ($d_0\leq d$) satisfying
\[
 \alpha(H_i)=\tau\alpha(H_i)\ \forall \alpha\in\Gamma_1 ,
\ i=1,\dots,d_0
\]
such that for all $x\in \g$
\[
 \delta(x)=\delta_\g(x)+\sum_{i=1}^{d_0}[H_i,x]\wedge t_i=
\ad_x\left(r-\sum_{i=1}^{d_0}H_i\wedge t_i\right)
\]
and a Lie coalgebra structure $\delta_V:V\to \Lambda^2V$
satisfying
\[
 \delta_V t_1=\cdots=\delta_V t_{d_0}=0.
\]
\end{prop}
\begin{rem}
In the notation of the above theorem, if $d_0=d$ then $\delta _V\equiv0$.
Notice that if dim $V>1$, the structure on $\l$ is coboundary if and only if $\delta_V\equiv 0$, which was already predicted in  item 2 of
lemma \ref{lema:favorable}. The examples
with $\delta_V\neq 0$ were not covered in \cite{Del}, since this work considers only
coboundary structures.
\end{rem}
Notice that the election of the $H_i$ appearing in the theorem above depends 
on a choice of a basis for
the complement of $\Ker(\DD)\subset V^*$. If one fixes a complement (of dimension $d_0$ in the notations of the theorem),
then the action of $\GL(d_0,\K)$ acts on the set of basis of this  complement, so we see that
$\GL(d_0,\K)$ acts on the set of $d_0$-uples $(H_1,\cdots, H_{d_0})$ in the obvious way, without
changing the isomorphism class of the Lie bialgebra $\l$. The case $d_0=1$ is corollary
\ref{coross1}. The following is an example for $\dim V=2$.

\begin{ex}
\label{exdim2}
Suppose that
 $\l=\g\times V$ is a product of a semisimple Lie algebra $\g$ and an abelian Lie algebra $V$ with $\dim V=2$; 
 write $V =\langle t_1,t_2\rangle$; then
the Lie bialgebra structures on $\l$ are of three possible types:
\begin{enumerate}
\item If $\DD=0$ then $\l=\g\times V$ is a product Lie bialgebra, \ie 
$$\delta (x+v)=\delta_{\g}(x)+\delta_{V}(v)$$ for any $x\in\g,v\in V$.
For any fixed Lie bialgebra structure $\delta_{\g}$ on $\g$, there are two isomorphism classes, namely, $\delta_V=0$, or $\delta_V\neq 0$, which is the unique non-coabelian
two dimensional Lie coalgebra.

\item If Im $\DD=\k D\neq 0$ , then
\[
\delta (x)=\delta_{\g}(x)+[H,x]\w t_1,\quad
\delta t_1=0,
\quad
\delta t_2=a t_1\w t_2
\]
with $\delta_\g$ a Lie cobracket on $\g$ and $H\in\Ker(\delta_g)$. Changing $H$ by a nonzero scalar multiple,
the isomorphism class of the Lie bialgebra does not change.
We may also assume $a=0$ or $1$. Notice that if $a=1$ then
the cobracket is {\em not} coboundary.

\item If Im $\DD=\k D_1\oplus \k D_2$ of dimension two,
$D_i=\ad_{H_i}, i=1,2$, then
$$\delta (x+v)=\delta_{\g}(x)+[H_1,x]\w t_1+[H_2,x]\w t_2
+\delta_V(v)$$
 for any $x\in\g,v\in V$,  with the following restrictions:
there exists $c =0,1$
such that $[H_1,H_2]=cH_1$
and the Lie coalgebra structure
$\delta_V$
is given by
$\delta_V t_1=ct_1\w t_2$, $\delta_V t_2=0$. Notice that if the Lie bialgebra structure
$\delta_\g$ on $\g$ is factorizable, then $c=0$ and hence $\delta$ is coboundary.
\end{enumerate}
\end{ex}


\begin{ex}{\em Cremmer-Gervais.}
Consider $\l=\gl(n,\K)=\sl(n,\K)\times \K$ with $\K=\R$ or $\C$. 
Fix the Cartan subalgebra $\h$ of traceless diagonal matrices
 and the factorizable Lie bialgebra
structure on $\sl(n,\C)$ given by an r-matrix $r$ with $r+r^{21}=\Omega$ and skewsymmetric component obtained from the discrete parameter $\Gamma _1=\{\alpha_1 ,\dots,\alpha_{n-2} \}$,  $\Gamma _2=\{\alpha_2 ,\dots,\alpha_{n-1} \}$ and $\tau(\alpha_i)=\alpha_{i+1},\ 1\leq i\leq n-2$, 
and any corresponding $\lambda\in\Lambda^2\h$.
As it was proved in \cite{AJ}, this BD-data on $\sl(n,\C)$ gives place to a factorizable Lie bialgebra
 structure on $\sl(n,\R)$, considered as its split form via the usual sesquilinear involution, 
if and only if $\lambda\in\Lambda^2_\C(\h)\cap \Lambda^2_\R\sl(n,\R)=\Lambda^2_\R(\h_\R)$, if we denote by $\h_\R$ 
the Cartan subalgebra of $\sl(n,\R)$ consisting of traceless real diagonal matrices.

The equations
\[
\alpha(H)=(\tau\alpha)(H)
\]
for all $\alpha\in\Gamma_1$, $H\in\h$, form a  system of $n-2$ equations in the $n-1$ variables
which are the coefficients of $H$ in the basis
$\{H_{\alpha _1},\dots,H_{\alpha _{n-1}}\}$ of $\h$; hence the space of solutions has dimension one.
In fact, we knew by other means that the regular element 
\[
H_r:=[\ ,\ ](r_\Lambda)=\underset{\alpha\in\Phi^+}{\sum}H_{\alpha}
\]
lies in $\Ker(\delta)$,
since 
$\D=[\ ,\ ]\circ\delta=\ad_{H_r}$ is a biderivation, in virtue of propositions
\ref{propD} and \ref{propbider}.
As a consequence, all biderivations  of $(\g,r)$ are scalar multiples of $\ad_{H_r}$.
On the other hand, analogous result holds in the real case, if we consider the subspace of $\h_\R$ of  real solutions. Notice that
$H_r=[\ ,\ ](r_\Lambda)\in\sl(n,\R)$. 

Both in the complex and in the real case, we conclude that there are exactly two isomorphism classes of Lie bialgebra on $\l$
such that $\l/V=(\g,r)$, given explicitly by
\[\delta_1(x+v)=\delta_\g(x)+D(x)\w t=\ad_x(r)+[H_r,x]\w t\]
and
\[\delta_2(x+v)=\delta_\g(x)=\ad_x(r)\]

\begin{ex}Let $\l=\gl(4,\C)=\sl(4,\C)\times \C$ and  $\l_0=\gl(4,\R)=\sl(4,\R)\times \R$, denote 
also $\g =\sl(4,\C)$ and  $\g_0 =\sl(4,\R)$.
Let $\Delta=\{\alpha,\beta,
\gamma\}$ a choice of simple roots  with respect to a root system for a  given  Cartan subalgebra $\h$ of $\g$. Recall that
a basis of root vectors  of $\g$ is
$$\mathcal B=\{x_\alpha,x_\beta,x_\gamma,
x_{\alpha+\beta},x_{\beta+\gamma},x_{\alpha+\beta+\gamma},
x_{-\alpha},x_{-\beta},x_{-\alpha -\beta},x_{-\beta-\gamma},x_{-\alpha-\beta-\gamma}\}\cup\{h_\alpha,h_\beta,h_\gamma\},$$
the Cartan matrix is 
$A=\left(
\begin{array}{rrrr}
2 & -1 & 0 \\
-1 & 2 & -1 \\
0 & -1 & 2  
\end{array}\right)$
and the Dynkin diagram is
$
\begin{picture}(75,0)
\put(10,15){\hbox{$\alpha$}}
\put(40,15){\hbox{$\beta$}}
\put(70,15){\hbox{$\gamma$}}
\put(10,5){\circle{10}}
\put(15,5){\line(1,0){20}}
\put(40,5){\circle{10}}
\put(45,5){\line(1,0){20}}
\put(70,5){\circle{10}}
\end{picture}
$.
\end{ex}
In  case of the empty BD-triple, all $H\in\h$ are solutions of $\tau\alpha(H)=\alpha(H)$. 
In the following table, we list (up to isomorphism of the Dynkin diagram) all possible
non-trivial discrete parameters for $\sl(4,\C)$ and generators of the space
of solutions
$\{H\in\h : \alpha(H)=(\tau\alpha)(H)\ \forall\alpha\in\Gamma_1\}$.
Notice that $\h=\Z_\g(H_0)$ (see proposition \ref{centr=cartan}) \ie
 the initial Cartan subalgebra coincides with the centralizer
of the regular element  $H_0=[\ ,\ ](r_\Lambda)$ explicitly given by
\[
H_0=
3h_\alpha+4h_\beta+3h_\gamma=
\sum\limits_{\alpha\in\Phi^+}H_{\alpha}
\]
\begin{center}
{\em $\Gamma_1$ and $\Gamma_2$ are subsets of $\Delta$ represented by the black roots.}
	\end{center}
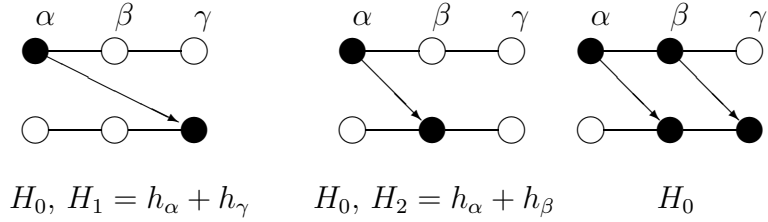
\[
\begin{array}{ccc}
\begin{picture}(110,55)
\put(20,45){\hbox{$\alpha$}}
\put(50,45){\hbox{$\beta$}}
\put(80,45){\hbox{$\gamma$}}

\put(20,35){\circle*{10}}
\put(25,35){\line(1,0){20}}
\put(50,35){\circle{10}}
\put(55,35){\line(1,0){20}}
\put(80,35){\circle{10}}

\put(20,5){\circle{10}}
\put(25,5){\line(1,0){20}}
\put(50,5){\circle{10}}
\put(55,5){\line(1,0){20}}
\put(80,5){\circle*{10}}

\put(20,35){\vector(2,-1){54}}
\end{picture}
&
\begin{picture}(80,55)
\put(20,45){\hbox{$\alpha$}}
\put(50,45){\hbox{$\beta$}}
\put(80,45){\hbox{$\gamma$}}

\put(20,35){\circle*{10}}
\put(25,35){\line(1,0){20}}
\put(50,35){\circle{10}}
\put(55,35){\line(1,0){20}}
\put(80,35){\circle{10}}

\put(20,5){\circle{10}}
\put(25,5){\line(1,0){20}}
\put(50,5){\circle*{10}}
\put(55,5){\line(1,0){20}}
\put(80,5){\circle{10}}

\put(20,35){\vector(1,-1){26}}
\end{picture}
&
\begin{picture}(80,55)
\put(20,45){\hbox{$\alpha$}}\put(50,45){\hbox{$\beta$}}\put(80,45){\hbox{$\gamma$}}
\put(20,35){\circle*{10}}\put(25,35){\line(1,0){20}}\put(50,35){\circle*{10}}\put(55,35){\line(1,0){20}}\put(80,35){\circle{10}}
\put(20,5){\circle{10}}\put(25,5){\line(1,0){20}}\put(50,5){\circle*{10}}\put(55,5){\line(1,0){20}}\put(80,5){\circle*{10}}
\put(20,35){\vector(1,-1){26}}\put(50,35){\vector(1,-1){26}}
\end{picture}
\\
\begin{picture}(70,30)
\put(-10,10){$H_0$, $H_1=h_\alpha+h_\gamma$}
\end{picture}
&
\begin{picture}(70,30)
\put(0,10){$H_0$, $H_2=h_\alpha+h_\beta$}
\end{picture}
&
\begin{picture}(30,10)
\put(20,10){$H_0$}
\end{picture}
\end{array}
\]
Indeed, we knew that the regular element $H_0$
lies in $\Ker(\delta)$ for $\delta$ comming from any choice of BD-triple,
because 
$\D=[\ ,\ ]\circ\delta=-\ad_{H_0}$ is a biderivation,
 in virtue of propositions
\ref{propD} and \ref{propbider},
 and it is independent of the BD-triple by inspection.

On the other hand, for the real case, 
$H_0=[\ ,\ ](r_\Lambda)\in\sl(4,\R)
$, then in particular, $\h _0:=\Z_{\g_0}(H_0)$ 
is a (real) Cartan subalgebra of $\g_0$.
For each data, it is only left to find the generators of the real space  of solutions of $\tau\alpha(H)=\alpha(H)$ for all $H\in\h_0$. 
Notice 
\[
\dim_\R\{H\in\h_0 : \alpha(H)=(\tau\alpha)(H)\ \forall\alpha\in\Gamma_1\}=\dim_\C\{H\in\h : \alpha(H)=(\tau\alpha)(H)\ \forall\alpha\in\Gamma_1\},
\]
\ie this real space 
is a real form of the complex space of solutions of the same equations viewed in $\h$.
\end{ex}

\begin{ex}{\em A non-triangular, non factorizable and not coboundary example.
}
Consider $\g=\su(2)\times \sl(2,\R)$,  $\l=\g\times\R^2$, $\{u_1,u_2,u_3\}$ a basis of $\su(2)$ with brackets $[u_i,u_j]=\sum_k\epsilon_{ijk}u_k$, where $\epsilon$ is
the totally antisymmetric symbol, and $\{h,x,y\}$ the standard basis of $\sl(2,\R)$. There are no nontrivial triangular structures in 
$\su(2)$ (see \cite{FJ}); moreover, all Lie bialgebra structures on $\su(2)$ are almost
factorizable and isomorphic to some positive multiple of  the coboundary associated to
$u_1\w u_2$. On the other hand, there are nontrivial triangular
structures in $\sl(2,\R)$, all of them isomorphic to the corresponding to $\pm h\w x$.
So, let us fix $r=u_1\w u_2+h\w x\in\Lambda^2\g$ and 
$\delta_\g(w)=\ad_w(r)$, for all $w\in\g$. 
In order to list {\em all} isomorphism classes of Lie bialgebra structures on $\l=\g\times\R^2$, we need to compute 
$\bider(\g,\delta_\g)$. Let 
\[
H_r=[-,-](r)=[u_1,u_2]+[h,x]=u_3+2x
\]
thus, by corollaries  \ref{coro:bider} and \ref{corobider}, we know that 
\[\bider(\g,\delta_\g)\cong\Ker\delta_\g\subseteq
\{w\in\g: [w,H_r]=0\}\]
For any $w=u+s\in\su(2)\times\sl(2,\R)$, we get
$ [w,H_r]=0\iff [u,u_3]=0 \hbox{ and } [s,x]=0$.
We conclude that $\bider(\g,\delta_g)$ is 2-dimensional, with
basis $\{\ad_{u_3},\ad_x\}$. In order to determine all possible
Lie bialgebra structures on $\l$ one may proceed as in example 
\ref{exdim2}. We ilustrate it showing only one possibility. Choose $\{t_1,t_2\}$ 
a basis of $\R^2$; if one defines
\[
\delta(w)=\ad_w(r)+[w,c_1 u_3+ c_2 x]\w t_1,\ 
\delta t_1=0,\
\delta t_2=t_1\w t_2
\]
for any $c_1,c_2\in\R$, then this structure is not coboundary, since $\delta|_{\R^ 2}\neq 0$.
We remark also that all non-coboundary stuctures on $\l$, such that induce
$\delta_\g$ on $\g$,
are of this form.
\end{ex}

\end{document}